\documentclass{amsart} 
\usepackage{enumerate}
\usepackage[pdftex, pdfusetitle, pdfpagelabels, plainpages=false,  bookmarks, bookmarksnumbered]{hyperref} 

\title{On constructions preserving the asymptotic topology of metric spaces}
\author{Gregory C. Bell}
\address{Mathematics and Statistics, University of North Carolina at Greensboro, Greensboro, NC 27412, USA}
\email{\href{mailto:gcbell@uncg.edu}{\nolinkurl{gcbell@uncg.edu}}}

\author{Danielle S. Moran}
\address{Mathematics, Guilford College, 5800 West Friendly Ave, Greensboro, NC 27410, USA}
\email{\href{mailto:morands@guilford.edu}{\nolinkurl{morands@guilford.edu}}}

\date{\today} 

\theoremstyle{plain}
\newtheorem{theorem}{Theorem}[section]
\newtheorem{lemma}[theorem]{Lemma}

\newtheorem{proposition}[theorem]{Proposition}
\newtheorem{conjecture}[theorem]{Conjecture}
\newtheorem{corollary}[theorem]{Corollary}
\theoremstyle{definition}

\newtheorem{question}[theorem]{Question}

\theoremstyle{remark}

\renewcommand{\epsilon}{\varepsilon}
\newcommand{\NN}{\ensuremath{\mathbb{N}}}
\newcommand{\RR}{\ensuremath{\mathbb{R}}}
\newcommand{\G}{\ensuremath{\Gamma}}
\newcommand{\g}{\ensuremath{\gamma}}
\newcommand{\s}[1]{\ensuremath{\mathcal{#1}}}

\newcommand{\sX}{\s{X}}
\newcommand{\sU}{\mathcal{U}}
\newcommand{\sV}{\mathcal{V}}
\newcommand{\sW}{\mathcal{W}}
\newcommand{\sY}{\mathcal{Y}}
\newcommand{\fG}{\mathfrak{G}}

\DeclareMathOperator{\as}{asdim}

\DeclareMathOperator{\diam}{diam}
\DeclareMathOperator{\lk}{lk}

\begin{document}

\begin{abstract} We prove that graph products constructed over infinite graphs with bounded clique number preserve finite asymptotic dimension. We also study the extent to which Yu's property A, Dranishnikov's property C, and Dranishnikov and Zarichnyi's straight finite decomposition complexity are preserved by constructions such as unions, free products, and group extensions.
\end{abstract}
\maketitle
\section{Introduction}

The asymptotic dimension of a metric space was introduced by Gromov \cite{Gr93} in his study of large-scale invariants of finitely generated groups. It is the large-scale analog of covering dimension in topology. Although interesting in its own right, asymptotic dimension gained the interest of the larger mathematical community following the work of Yu \cite{Yu98}. Yu showed that a finitely generated group with finite asymptotic dimension satisfies the famous Novikov higher signature conjecture. This generated interest in determining whether the asymptotic dimension of various groups and classes of groups is finite. Although later work (e.g. \cite{Yu00,KY}, among others) has refined the technology to determine whether a group satisfies the Novikov (or related) conjectures, there is still a great deal of interest in this simple large-scale invariant.

The asymptotic dimension is a coarse invariant. 
Each countable group can be endowed with a proper left-multiplication invariant metric that is unique up to coarse equivalence (see section 2). This means that the large-scale invariants associated to a group with such a metric are group invariants and are independent of choices involved in determining the specific metric. This also makes the class of countable groups a natural one for the study of asymptotic invariants. On the other hand, proper left-invariant metrics are not always natural, e.g. $\mathbb{Q}$ with its usual metric is not proper.

The class of groups with finite asymptotic dimension is vast and contains (among many others) hyperbolic groups \cite{Ro05}, nilpotent groups \cite{BD-Hur}, solvable groups with rational Hirsch length \cite{Dr-Smi}, Coxeter groups \cite{DJ}, mapping class groups \cite{BBF}, and groups admitting a proper isometric action on finite dimensional CAT(0)-cube complexes \cite{Wright}. 
Moreover this class is closed under the operations of (finite) direct product, free products with amalgamation, and group extensions, see \cite{BD-Hur}. Recently, Antol\'in and Dreesen \cite{An-Dre} computed a formula for the asymptotic dimension of a graph product of groups using results of \cite{Dr-amalg} and \cite{Green}. The first main result of this paper is to extend this result to certain graph products over infinite graphs: Theorem \ref{Asdim-Infinite-Graphs}. Because these graphs are infinite, the techniques used by Antol\'in and Dreesen are not applicable. Instead we exploit the structure of these graph products to explicitly construct the covers from the definition of asymptotic dimension at each scale $R$.

Dranishnikov, Keesling and Uspenskij \cite{DKU} showed that $\as\mathbb{Z}^n=n$, so any group that contains a copy of $\mathbb{Z}^n$ for each $n$ will necessarily have infinite asymptotic dimension. Such groups are not difficult to construct (for example, see \cite{Ro03}). For groups and spaces whose asymptotic dimension may be infinite, one can consider other dimension-like coarse invariants, such as asymptotic property C, finite decomposition complexity, straight finite decomposition complexity, or property A. For metric spaces with bounded geometry, finite asymptotic dimension implies both asymptotic property C and finite decomposition complexity. Both of these notions imply straight finite decomposition complexity. Finally, spaces with straight finite decomposition complexity have Yu's property A. See \cite{Gold} for a nice summary of these implications. 

The second goal of this paper is to apply the techniques of \cite{An-Dre} to point out that graph products of groups with property A have property A. Although, we would like to extend this result to asymptotic property C, it is not clear that it does extend. In particular, the standard approach to such properties breaks down completely in the case of property C. We were only able to show that asymptotic property C is preserved by certain infinite unions (Theorem \ref{C-Union}). We cannot show that it is preserved by amalgamated free products or direct products. If amalgams and direct products could be shown to preserve asymptotic property C, then the techniques of \cite{An-Dre} could be applied to show that graph products preserve it. In light of the result of Pol and Pol \cite{Pol-Pol}, it is conceivable that there is some space (or even a group) with asymptotic property C whose square does not have asymptotic property C.

The final goal of the paper is to prove some permanence results for straight finite decomposition complexity along the lines of those shown in \cite{GTY1}. 

The paper is organized as follows. In the next section we recall some of basic facts and give precise definitions. In Section 3 we state and prove the main theorem concerning infinite graph products and asymptotic dimension. We also show that property A is preserved by finite graph products. In the fourth section, we prove that asymptotic property C is preserved by certain infinite unions and state some open questions concerning asymptotic property C. The permanence properties of straight finite decomposition complexity appear in the final section. It should be noted that many of the goals of this paper align with the excellent survey by Guentner \cite{Guentner-Survey}.

The authors wish to thank the anonymous referees of this paper for careful reading, corrections, and for helping to clarify several points.

\section{Preliminary notions and definitions}

Let $(X,d_X)$ and $(Y,d_Y)$ be metric spaces. Recall that a map $f:X\to Y$ is called \textit{proper} if the preimage of every compact set is compact. A metric is called \textit{proper} if the distance function is a proper map, i.e., if closed balls are compact.
A function $f:X\to Y$ is called {\em uniformly expansive} if there is a non-decreasing $\rho_2:[0,\infty)\to[0,\infty)$ such that 
\[d_Y(f(x),f(x'))\le\rho_2(d_X(x,x')).\]
The function $f:X\to Y$ is called {\em effectively proper} if there is some proper, non-decreasing $\rho_1:[0,\infty)\to[0,\infty)$ such that
\[\rho_1(d_X(x,x'))\le d_Y(f(x),f(x')).\]
The function $f:X\to Y$ is called a {\em coarsely uniform embedding} if there exist functions $\rho_1,\rho_2:[0,\infty)\to[0,\infty)$ such that, \[\rho_1(d_X(x,x'))\le d_Y(f(x),f(x'))\le \rho_2(d_X(x,x'))\] and $\rho_1\to\infty$. The spaces $X$ and $Y$ are said to be {\em coarsely equivalent} if there is a coarsely uniform embedding of $X$ to $Y$ and there is some $R>0$ so that for each $y\in Y$ there is some $x\in X$ so that $d_Y(f(x),y)\le R$. When the $\rho_i$ can be taken to be linear, $f$ is called a {\em quasi-isometric embedding} and the corresponding equivalence is {\em quasi-isometry}. 

Let $R>0$ be a (large) real number. A collection $\sU$ of subsets of the metric space $X$ is said to be \textit{uniformly bounded} if there is a uniform bound on the diameter of the sets in $\sU$; a collection $\sU$ is said to be \textit{$R$-disjoint} if, whenever $U\neq U'$ are sets in $\sU$, then $d(U,U')>R$, where $d(U,U')=\inf\{d(x,x')\mid x\in U, x'\in U'\}$. A family that is uniformly bounded and $R$-disjoint will be called \text{$R$-discrete}. Gromov \cite{Gr93} describes this situation by saying that $\cup_{U\in\sU} U$ is $0$-{\em dimensional on} $R$-{\em scale}. 

We say that the {\em asymptotic dimension} of the metric space $X$ does not exceed $n$ and write $\as X\le n$ if for each (large) $R>0$, $X$ can be written as a union of $n+1$ sets with dimension $0$ at scale $R$. 
There are several other useful formulations of the definition (see \cite{BD-AD}) but we shall content ourselves with this one.

Yu defined property A for discrete metric spaces as a generalization of amenability of groups \cite{Yu00}. A discrete metric space $X$ has {\em property A} if for any $r>0$ and any $\epsilon>0$, there is a collection of finite subsets $\{A_x\}_{x\in X}$, where $A_x\subset X\times\NN$, so that 
\begin{enumerate}
	\item $(x,1)\in A_x$ for each $x\in X$;
	\item for every pair $x$ and $y$ in $X$ with $d(x,y)<r$, $\frac{|A_x\Delta A_y|}{|A_x\cap A_y|}<\epsilon$; and 
	\item there is some $R$ so that for every $n$ such that $(y,n)\in A_x$, $d(x,y)\le R$.
\end{enumerate}

The asymptotic analog of Haver's property C for metric spaces was defined by Dranishnikov \cite{Dr00}. We say that a metric space $X$ has {\em asymptotic property C} if for any given number sequence $R_1\le R_2\le R_3\le\cdots$ there exists some integer $n$ and a cover of $X$ that can be decomposed into $n$ uniformly bounded families $\sU^1,\sU^2,\ldots,\sU^n$ in such a way that each $\sU^i$ is $R_i$-disjoint (and $\cup_{i=1}^n\sU^i$ covers $X$).  It is clear that a metric space with finite asymptotic dimension will have asymptotic property C. Dranishnikov showed that a discrete metric space with bounded geometry and asymptotic property C also has property A, see \cite[Theorem 7.11]{Dr00}. 

In two papers \cite{GTY1,GTY2} Guentner, Tessera and Yu defined another coarse invariant of groups that is applicable when the asymptotic dimension is infinite: finite decomposition complexity. Following this, Dransihnikov and Zarichnyi defined a related notion: straight finite decomposition complexity. By way of notation, we write $A=\displaystyle\sqcup_{\hbox{\tiny$R$-disjoint}}A_i$ to mean that the subset $A$ can be decomposed as a union of sets $A_i$ in such a way that $d(A_i,A_j)>R$ whenever $i\neq j$.
Let $\sX$ and $\sY$ be families of metric spaces. For a positive $R$, we say that $\sX$ is $R$-decomposable over $\sY$ and write $\sX\xrightarrow{R}\sY$ if for any $X\in\sX$ one can write 
\[X=Y^0\cup Y^1\, \hbox{ where } Y^i=\bigsqcup_{\tiny R\hbox{-disjoint}} Y^{ij},\hbox{ for i=0,1,}\]
where the sets $Y^{ij}\in\sY$.

We begin by describing the metric decomposition game for $X$. In this game two players take turns. First, Player 1 asserts a number $R_1$. Player 2 responds by finding a metric family $\sY^1$ and a $R_1$-decomposition of $\{X\}$ over $\sY^1$. Then, Player 1 selects a number $R_2$ and Player 2 again finds a family $\sY^2$ and an $R_2$-decomposition of $\sY^1$ over $\sY^2$. Player 2 wins if the game ends in finitely many steps with a family that consists of uniformly bounded subsets. The metric space $X$ is said to have {\em finite composition complexity} or FDC, if there is a winning strategy for Player 2 in the metric decomposition game for $X$, see \cite{GTY1}.

Let $\sX$ be a family of metric spaces. We say that the metric family $\sX$ has {\em straight finite decomposition complexity} sFDC if for every sequence $R_1\le R_2\le \cdots$ there exists an $n$ and metric families $\sY^i$ ($i=1,2,3,\ldots, n)$ so that with $\sX=\sY^0$, $\sY^{i-1}\xrightarrow{R_i}\sY^i$ for $i=1,2,3,\ldots, n$, and such that $\sY^n$ is uniformly bounded, see \cite{DZ}. The metric space $X$ will be said to have sFDC if the family $\{X\}$ does. It is clear that by restricting the families, this property can be seen to pass to subsets.

A finitely generated group with generating set $S=S^{-1}$ can be endowed with a left-invariant metric called the {\em word metric} by taking $d_S(g,h)=\|g^{-1}h\|_S$, where the norm $\|\gamma\|_S$ is zero at the identity and otherwise is the length of a shortest $S$-word that presents $\gamma$. It is easy to see that if $S$ and $S'$ are finite generating sets on the finitely generated group $\G$, then the metric spaces $(\G,d_S)$ and $(\G,d_{S'})$ are quasi-isometric. 

The situation for non-finitely generated groups is less clear. Ideally, one would like to endow any countable group with a metric structure that is an invariant of coarse isometry. Smith showed that on a countable group any two left-invariant, proper metrics are coarsely equivalent  \cite{Smi}. Moreover he shows that a weight function (defined below) on a countable group induces a left-invariant, proper metric. By a weight function on a generating set $S=S^{-1}$ for a group, we mean a function $w:S\to \overline{\RR}_+$ for which 
\begin{enumerate}
	\item if $w(s)=0$ then $s$ is the group identity $e$;
	\item $w(s)=w(s^{-1})$; and 
	\item for each $N\in\NN$, $w^{-1}([0,N])$ is finite.
\end{enumerate}

One then defines a norm by $\|\gamma\|=\inf\{\sum w(s_i)\mid x=s_1s_2\cdots s_n\}$, where the norm of the identity is defined to be $0$ (i.e. it is presented by the empty product).

Let $\G$ be an undirected graph without loops or multiple edges. Let $V(\Gamma)$ and $E(\G)$ be the set of vertices and edges of $\G$, respectively. Suppose that $\fG=\{G_v\}$ is a collection of groups indexed by the elements of $V(\G)$. The graph product $\G\fG$ of the collection $\fG$ over the graph $\G$ is defined to be the free product of the $G_v$ with the additional relations that whenever $\{v,v'\}$ is an edge in $\G$, then $gg'=g'g$ for all $g\in G_v$ and $g'\in G_{v'}$. Thus, if $E(\Gamma)=\emptyset$, then $\G\fG$ is the free product of the groups $G_v$, i.e. $\ast G_v$. If $\Gamma$ is the complete graph on $n$ vertices, we obtain the direct product $\G\fG=G_{v_1}\times\cdots\times G_{v_n}$. Graph products were introduced in \cite{Green}.

We will often refer to a word in a graph product (or free product) as being expressed in {\em syllables}. We say that the word $g_1\cdots g_\ell$ is an expression of $g$ in syllables if each  $g_i$ is a reduced, nontrivial word in some single vertex group, and no two consecutive $g_i$ and $g_{i+1}$ belong to the same vertex group.

Let $\G$ be a countable graph and let $\fG=\{G_v\}_{v\in V(\G)}$ be a collection of countable groups indexed by the vertices of $\G$. We may endow $\G\fG$ with a proper left-invariant metric by choosing generating sets $S_v$ for each group and assigning a weight function $w:\sqcup S_v\to \mathbb{N}$.
For each $r\in \mathbb{N}$, define a graph $\G_r$ by taking the collection of vertices $v\in V(\G_r)$ to be precisely those vertices $v$ for which some element of $S_v$ is assigned a weight $\le r$. Thus, outside of this vertex set, all weights exceed $r$. An edge connects two vertices of $\G_r$ if and only if there is an edge in $\G$ connecting the corresponding vertices of $\G$.

We will say that an element $x\in \G\fG$ is $r$-\textit{permissible} (or simply \textit{permissible} when $r$ is understood) if no reduced word presenting the element $x$ can be made to end with a non-trivial element of a group $G_v$ with $v\in\G_r$. Thus, the group identity is $r$-permissible for all $r$. By way of notation, let $\G_r\fG$ denote those element of $\G\fG$ that can be expressed in terms of $G_v$ with $v\in\G_r$.

Finally, for a graph $\G$, we recall that the {\em clique number} $\omega(\G)$ is the maximum number of vertices in a clique in $\G$; i.e., the size of the largest set of vertices for which each pair is connected by an edge in $\G$. 

\section{Asymptotic dimension of graph products} 

In this section we extend the result of Antol\'in and Dreesen concerning asymptotic dimension of graph products of groups in two directions. First, we extend the asymptotic dimension result to include certain infinite graphs. Second, we show that one can replace finite asymptotic dimension everywhere with property A and arrive at the corresponding conclusion. An anonymous referee suggested an alternate approach to ours: one can apply the result of Antol\'in and Dreesen \cite{An-Dre} and a result of Dranishnikov and Smith \cite{Dr-Smi}, which shows that the asymptotic dimension of a countable group is the supremum of the asymptotic dimensions of its finitely generated subgroups and then use the fact that a graph product on a countable graph is obtained as a union over subgraphs. Instead, we pursue a more elementary approach.

\begin{lemma}\label{Lemma3.1} Let $\G \fG$ be a graph product of countable groups $\fG=\{G_v\}$ over a countable graph $\G$ with a proper metric given by a weight function as described above. Let $r>0$ be given and take $\G_r$ as above. Then, each element of $\G \fG$ can be written in the form $xb$, where $x$ is permissible and $b\in \G_r\fG$. Moreover, if $x\neq x'$ are permissible, then $d(xb,x'b')>r$.
\end{lemma}

\begin{proof} First, we check that each element has such a form. To this end, let $g\in \G\fG$ be given and write $g=g_1\cdots g_t$ as an expression in syllables. We proceed by induction on the number of syllables $t$. If $t=1$, then either $g_1$ is in $\G_r\fG$ or not. In the first case, it can be written as $xg_1$, where $x=e$. In the latter case, $x=g_1$ is permissible.

Suppose now that every word of syllable length at most $t-1$ can be written in the form $xb$ with $x$ permissible and $b\in\G_r\fG$. Then, consider $g=g_1\cdots g_t$. Since $g_1\cdots g_{t-1}$ has syllable length shorter than $t$ it can be written in the form $xb$. Therefore, express $x$ and $b$ in syllables so that we have $g=x_1\cdots x_p b_{p+1}\cdots b_{t-1}g_t$. If $g_t$ itself is in $\G_r\fG$, then this word is already in permissible form. 

Suppose therefore, that $g_t\notin \G_r\fG$. If it commutes with $b_{t-1}$, then we can write $b_{t-1}g_t=g_tb_{t-1}$ and therefore we have $g=x_1\cdots b_{t-1}g_t=x_1\cdots g_tb_{t-1}$. Now, since its length is less than $t$, the element $x_1\cdots g_t$ can be written as some $x'b'$ in permissible form. But, then $g=x'b'b_{t-1}$ is a permissible presentation of $g$. 

Finally, we consider the case in which $g_t$ does not commute with $b_{t-1}$. If any rearrangement of this word allows $g_t$ to commute past a syllable, then we apply the argument of the preceding paragraph to obtain a word in permissible form. Otherwise, $x=g$ is already permissible. 

Now, we show the disjointness condition holds. Suppose that $x$ and $x'$ are distinct, but permissible. Then, write $x^{-1}x'=z$ for some $z\in \G \fG$. Observe that $z\notin \G_r\fG$, as, if it were, then $xz$ would be a presentation of $x'$ that ends with a non-trivial element of $\G_r\fG$, which is not allowed. Thus, $z$ must contain some element that is not in $\G_r\fG$. Hence it contains a generator $s$ from a group with weight $>r$. Thus, $d(xb,x'b')=\|b^{-1}zb'\|\ge\|s\|>r$.
\end{proof}

\begin{theorem}\cite[Theorem 6.3]{An-Dre} Let $\Gamma$ be a finite simplicial graph and let $\fG$ be a family of finitely generated groups indexed by vertices of $V(\G)$. Let $G=\G\fG$. Let $\mathcal{C}$ be the collection of subsets of $V(\G)$ spanning a complete graph. Then
\[\as G\le \max_{C\in\mathcal{C}}\sum_{v\in C}\max(1,\as G_v).\]
\end{theorem}

For our present purposes, we need a slightly weaker result that we state as a corollary.

\begin{corollary}\label{Cor3.3} Let $\Gamma$ be a finite simplicial graph with $\omega(\G)\le k$ and let $\fG$ be a collection of finitely generated groups indexed by $v\in V(\G)$ such that $0<\as G_v\le n$ for all $v\in V(\G)$. Then, $\as \G\fG\le nk$. 
\end{corollary}

\begin{proof} We have that $\max(1,\as G_v)=\as G_v$ for each $v$. Also, there is at least one $C\in\mathcal{C}$ with $\omega(\Gamma)$ elements. Thus, 
\[\as G\le \max_{C\in\mathcal{C}}\sum_{v\in C}\max(1,\as G_v)\le\omega(\G)\max_{v\in V(\G)}\{\as G_v\}\le kn.\]
\end{proof}

Really, all that is necessary for the preceding proof to work is that at least one of the $G_v$ should be infinite, forcing $n>0$. If all $G_v$ are finite, then $\as G\le k$ instead of the estimate given above, which would be $0=nk$. 

\begin{theorem} \label{Asdim-Infinite-Graphs} Let $\Gamma$ be a countable graph with clique number $\omega(\G)\le k$. Suppose that $\{G_v\}_{v\in V(\G)}$ is a collection of countable groups (in proper metrics) with $0<\as G_v\le n$ for all $v\in V(\G)$. Then, in a left-invariant proper metric, $\as \G\fG\le nk.$
\end{theorem}

\begin{proof} For a given $r>0$ we will construct a cover by $nk+1$ uniformly bounded, $r$-disjoint families of subsets of $\G\fG$. Since $\G\fG$ is a countable group that is not finitely generated, we endow it with a metric arising from a weight function $\bar{w}:V(\G)\to\NN$ as described above. 

Define a subgraph $\G_r$ of $\G$ by setting $V(\G_r)=\bar{w}^{-1}([0,r])$ and by defining an edge between two vertices of $\G_r$ if and only if there is an edge between these vertices in $\G$. By Corollary \ref{Cor3.3}, we know that $\as\G_r\fG\le nk$. Thus, there is a cover by $nk+1$ $r$-disjoint families of uniformly bounded sets, say $\sU^0,\sU^1,\ldots, \sU^{nk}$. Let $P\subset \G\fG$ denote the set of all $\G_r$-permissible elements. 

For each $i$ define the collection $\{xU\mid x\in P, U\in\sU^i\}$. We claim that for each $i$, the collection is $r$-disjoint and uniformly bounded. Moreover, we claim that the union of these collections covers $\G\fG$. 

Let $x\in P$. Since the metric on $\G\fG$ is left-invariant, we know that $d(xu,xu')=d(u,u')$, for all $xu$ and $xu'$ in $xU$. Since $\diam(U)$ is uniformly bounded, we have that $\diam(xU)$ is also uniformly bounded.

Next, suppose that $xU$ and $x'U'$ are distinct sets, where $x,x'\in P$ and  $U, U'\in\sU^i$. If $x=x'$, then we have $d(xU,x'U')=d(xU,xU')=d(U,U')$, and since these sets must be different (yet still in the same family $\sU^i$), they are at least $r$-disjoint. If $x\neq x'$, then by Lemma \ref{Lemma3.1} $d(xu,x'u')>r$ and so these two families are $r$-disjoint. 

Finally, we show that the collection of all such families covers $\G\fG$. To this end, let $g\in\G\fG$ be given.  Then, by Lemma \ref{Lemma3.1} $g=xb$, where $x\in P$ and $b\in \G_r\fG$. Thus, there is some $i$ and some $U\in\sU^i$ so that $b\in U$. Thus, $g\in xU$, as required. 
\end{proof}

The following result and proof follow are similar to \cite[Theorem 6.3]{An-Dre}. 

\begin{proposition}\label{Cor-PropA} Let $\G$ be a finite graph and let $\fG=\{G_v\}_{v\in V(\G)}$ be a collection of countable groups with proper left-invariant metrics. If all the $G_v$ have property A, then $\G\fG$ has property A.
\end{proposition}
\begin{proof}
We proceed by induction on $|V(\Gamma )|$.  We note that if $|V(\Gamma )| = 1$, then $\G\fG = G_v$ which is assumed to have property A.

Now we suppose that $|V(\Gamma )| = n > 1$ and also that the theorem holds for graphs with fewer than $n$ vertices.

Then let $v\in V(\Gamma )$ be any vertex, and put $A = \{v\} \cup \lk(v), B = \Gamma - \{v\}, C = \lk(v)$.  Then, by Green \cite{Green} we have that $\G\fG = G_A\ast_{G_C}G_B$.

Now, we have two cases.  In the first case, $A = \Gamma$.  Then, since $v$ is connected to each vertex of $C$ and this encompasses all vertices of $\Gamma$, we have that $\G\fG = G_v\times G_C$. Now $G_v$ has property A by assumption. Since $|V(C)| < |V(\Gamma )|$ the induction hypothesis implies that $G_C$ has property A.  Since property A is preserved by direct products (by \cite{Yu00}), $\G\fG$ has property A.

In the second case, where $A \neq \Gamma$, we have then that $|V(A)| < |V(\Gamma )|$.  By definition, we have that $|V(B)| < |V(\Gamma )|$.  And so, by our induction hypothesis, $G_A$ and $G_B$ both have property A.  Since amalgamated free products preserve property A (see \cite{Dye,Tu,Be}), we conclude that $\G\fG$ has property A. 
\end{proof}

We end this section with an open problem, which we phrase as a conjecture. The techniques that we use in this paper cannot easily be applied in this situation.

\begin{conjecture} Let $\Gamma$ be a countably infinite graph and suppose that all $G_v\in\fG$ have property A. Then in a proper, left-invariant metric, $\G\fG$ has property A.
\end{conjecture}

\section{Asymptotic Property C}

The goal of this section is to show that asymptotic property C is preserved by some infinite unions. 

 We consider the case where $X$ can be expressed as a union of a collection of spaces with uniform asymptotic property C with the additional property that for each $r>0$ there is a ``core'' space such that removing this core from the families leaves the families $r$-disjoint. We begin by stating some results from \cite{BD1}.

Let $\sU$ and $\sV$ be families of subsets of a metric space $X$. Let $V\in\sV$ and $d>0$. Let $N_d(V;\sU)$ be the union of $V$ and the set of all $U\in\sU$ such that $d(V,U)\le d$. The {\em $d$-saturated union of $\sV$ in $\sU$} is the set $\sV\cup_d\sU=\{N_d(V;\sU)\mid V\in\sV\}\cup\{U\in\sU\mid d(V,U)>d\ \forall V\in\sV\}$. Note that (in general) $\sV\cup_d\sU\neq\sU\cup_d\sV$ and that $\emptyset\cup_d\sU=\sU=\sU\cup_d\emptyset$.

\begin{proposition}\cite[Proposition 2]{BD1} Assume that $\sU$ is a collection of subsets of a metric space $X$ that is $d$-disjoint and $R$-bounded, with $R\ge d$. Assume that $\sV$ is a collection of subsets that is $5R$-disjoint and uniformly bounded. Then, $\sV\cup_d\sU$ is $d$-disjoint and uniformly bounded.
\end{proposition}

Let $\{X_\alpha\}_{\alpha}$ be a family of metric spaces. We will say that the family $X_\alpha$ satisfies asymptotic property C {\em uniformly in} $\alpha$ if for every sequence $R_1<R_2<\cdots$ there exists an $n$ and $B_1<B_2<\cdots<B_n$ so that for each $\alpha$ there exist families $\sU^i_\alpha$ of $R_i$-disjoint, $B_i$-bounded families ($i=1,\ldots,n$) so that $\cup_{i=1}^n\sU^i_\alpha$ covers $X_\alpha$.

\begin{theorem} \label{C-Union} Suppose that $X=\cup_\alpha X_\alpha$ is a countable union of spaces that have uniform asymptotic property C. Suppose further that for each $r>0$ there is a $Y_r\subset X$ so that $Y_r$ has asymptotic property C and such that the family $\{X_\alpha-Y_r\}_\alpha$ is $r$-disjoint. Then, $X$ has asymptotic property C.
\end{theorem}

\begin{proof} Let $d_1<d_2<\cdots$ be a sequence of positive numbers. For each $\alpha$, choose families $\sU^i_\alpha$ of $d_i$-disjoint, $R_i$-bounded sets, $i=1,2,\ldots,n$. Since $R_i$ are upper bounds on diameters, we may take them to be increasing and insist that $R_i\ge d_i$. Put $r=5R_n$. Take $Y_r$ as in the statement of the theorem. 

Let $\sV^1,\sV^2,\ldots,\sV^k$ be $5R_i$-disjoint, $B_i$-bounded families of sets whose union covers $Y_r$.

Let $\overline{\sU_\alpha^i}$ denote the restriction of $\sU^i_\alpha$ to $X_\alpha-Y_r$. Next, put $\overline{\sU^i}=\cup_\alpha\sU^i_{\alpha}$. Note that $\overline{\sU^i}$ is $R_i$-bounded and $d_i$ disjoint. Finally, set $\sW^i=\sV^i\cup_{d_i}\overline{\sU^i}$, for $i=1,2,\ldots,\max\{k,n\}$. Here, we take $\sV^i=\emptyset$ or $\sU^i=\emptyset$ if $i>k$ or $i>n$, respectively. Thus, in these cases, we have $\sW^i=\overline{\sU^i}$ or $\sW^i=\sV^i$, respectively. By the above proposition, $\sW^i$ is $d_i$-disjoint and uniformly bounded. 

Finally, we show that the collection $\{\sW^i\}$ covers $X$. To this end, suppose that $x\in X$ is given. Suppose first that $x\in Y_r$. Then, since the collection $\{\sV^i\}_{i=1}^k$ covers $Y_r$, there is some $i_0$ so that $x\in V_0\in\sV^{i_0}$. Now, since $\sW^{i_0}$ contains the set $N_d(V_0;\overline{\sU}^{i_0})$, we see that every element of $V_0$ is in some set inside of $\sW^{i_0}$. Thus, in this case, $x$ is covered by some set in $\{\sW^i\}_{i=1}^{\max\{k,n\}}$.

Next, suppose that $x\notin Y_r$. Then, there is some $j_0$ and an $\alpha_0$ so that there is some $U_0\in\sU^{j_0}_{\alpha_0}$ that contains $x$. Either $d(U_0,\sV^{j_0})\le d$ or $d(U_0,\sV^{j_0})>d$. In the former case, we see that all elements of $U_0$ will be in some element of the type $N_d(V,\sU^{j_0}_{\alpha_0})$. Thus, this $x$ is in some element of $\{\sW^i\}$. In the latter case, $U_0$ is among the collection $\{U\in\overline{\sU}^{j_0}_{\alpha_0}\mid d(V,U)>d\ \forall\ V\in \sV^{j_0}\}$ and thus, is in some element of $\{\sW^i\}$.
\end{proof}

We end this section with some open problems.

\begin{question} Is asymptotic property C preserved by free products or amalgamated free products?
\end{question}

Note that in a recent preprint, Beckhardt \cite{Beck} has shown that a group acting by isometries on a metric space with asymptotic property C in such a way that the stabilizers have finite asymptotic dimension will have property C. This result is related to this question since the original proof that finite asymptotic dimension was preserved by amalgamated products used the action of the product on its Bass-Serre tree. This action is by isometries, the stabilizers have finite asymptotic dimension and the tree on which the space acts has finite asymptotic dimension.

\begin{question}
Is asymptotic property $C$ preserved by direct products?
\end{question}

If the answers to the previous two questions are both yes, then it would immediately follow that the following question also has a positive answer.

\begin{question} Let $\G$ be a finite graph. If all the $G_v$ have asymptotic property $C$, does $\fG\G$ have asymptotic property $C$?
\end{question}

If the answer to that question is yes, one could additionally ask the following.

\begin{question} Let $\Gamma$ be a countably infinite graph with bounded clique number. Suppose that all $G_v$ have asymptotic property $C$. Then, in a proper, left-invariant metric, does $G$ have asymptotic property $C$?
\end{question}
												
\section{Straight Finite Decomposition Complexity}

The goal of this section is to apply the techniques of Guentner, Tessera and Yu \cite{GTY1,GTY2} to the notion of straight finite decomposition complexity defined by Dranishnikov and Zarichnyi \cite{DZ}. It is shown there that sFDC is a coarse invariant, is preserved by finite unions, and is preserved by some infinite unions (analogous to our theorem above about property C). We extend these results to show that sFDC is preserved by fiberings and conclude that it is preserved by amalgamated products and graph products.

We begin by recalling some of the results from \cite{DZ}.

\begin{theorem} \cite[Theorem 3.1]{DZ} \label{DZ-ce} If $f:X\to Y$ is a coarse equivalence and if $Y$ has sFDC, then so does $X$.
\end{theorem}

We include a proof for the reader's convenience and also because we will use the same technique to prove our fibering theorem, Theorem \ref{Fibering}.

\begin{proof} Let $f:X\to Y$ be uniformly expansive and effectively proper. Suppose that $\rho:[0,\infty)\to [0,\infty)$ is an increasing function for which $d(f(x),f(x'))\le\rho(d(x,x'))$ for all $x$ and $x'$ in $X$. 

Let $R_1<R_2<\cdots$ be given and set $S_i=\rho(R_i)$ for each $i$. By way of notation, put $\{Y\}=\sV^0$. Then, since $Y$ has sFDC, there is some $m\in\NN$ and metric familes $\sV^1,\sV^2,\ldots,\sV^m$ so that $\sV^0\xrightarrow{S_1}\sV^1\xrightarrow{S_2}\sV^2\xrightarrow{S_3}\cdots\xrightarrow{S_m}\sV^m$ with $\sV^m$ bounded. According to \cite[Lemma 3.1.1]{GTY1}, if $\sV^{i-1}\xrightarrow{S_i}\sV^i$ then $f^{-1}(\sV^{i-1})\xrightarrow{R_i}f^{-1}(\sV^i)$. 

More explicitly, write $Y=V^1_0\cup V^1_1$, where \[V^1_i=\bigsqcup\limits_{\tiny S_1\hbox{-disjoint}}V^1_{ij},\] and $V^1_{ij}\in\sV^1$. Then $X=f^{-1}(Y)=f^{-1}(V^1_0)\cup f^{-1}(V^1_1)$, with 
\[f^{-1}(V^1_i)=\bigsqcup\limits_{\tiny R_1\hbox{-disjoint}}f^{-1}(V^1_{ij}).\] 

Then, for each $V\in\sV^1$, write $V=V^2_0\cup V^2_1$ where \[V^2_i=\bigsqcup\limits_{\tiny S_2\hbox{-disjoint}}V^2_{ij},\] and $V^2_{ij}\in\sV^2$. Then, as above, obtain an $R_2$-decomposition of $f^{-1}(\sV^1)$ over $f^{-1}(\sV^2)$. We continue in this way until we eventually find an $R_m$-decomposition of $f^{-1}(\sV^{m-1})$ over $f^{-1}(\sV^m)$. Since $f$ is effectively proper and $\sV^m$ is bounded, we apply \cite[Lemma 3.1.2]{GTY1} to conclude that $f^{-1}(\sV^m)$ is bounded, as required.
\end{proof}

Next, we obtain a version of \cite[Theorem 3.1.4]{GTY1} for straight finite decomposition complexity.

\begin{theorem}\label{Fibering} Let $X$ and $Y$ be metric spaces and let $f:X\to Y$ be a uniformly expansive map. Assume that $Y$ has sFDC and that for every bounded family $\sV$ in $Y$, the inverse image $f^{-1}(\sV)$ has sFDC. Then, $X$ has sFDC. 
\end{theorem}

\begin{proof} Let $R_1<R_2<\cdots$ be given. Set $\{Y\}=\sV^0$. Since $Y$ has straight finite decomposition complexity, and since $f$ is uniformly expansive, we take $S_i=\rho(R_i)$ as in Theorem \ref{DZ-ce} to find an $m$ and families $\sV^1,\sV^2,\ldots,\sV^m$ so that $\sV^{i-1}\xrightarrow{S_i}\sV^i$ and for which $\sV^m$ is bounded. Then, as before, we pull these families back to $X$ to obtain $f^{-1}(\sV^{i-1})\xrightarrow{R_i}f^{-1}(\sV^i)$. Since we assume that $f^{-1}(\sV^m)$ has straight finite decomposition complexity, we take the sequence $R_{m+1},R_{m+2},\dots$ and find $n$ and families $\sU^{m+1},\sU^{m+2},\ldots,\sU^{m+n}$ so that $f^{-1}(\sV^m)\xrightarrow{R_{m+1}}\sU^{m+1}$; $\sU^{m+j-1}\xrightarrow{R_{m+j}}\sU^{m+j}$ for $j>1$; and such that  $\sU^{m+n}$ bounded. Then, with $\sU^i=f^{-1}(\sV^i)$ for $i=0,1,\ldots, m$ we have $\sU^{i-1}\xrightarrow{R_i}\sU^{i}$ for all $i=1,2,\ldots,m+n$ as required.
\end{proof}

\begin{proposition} Let $G$ be a countable group expressed as a union of subgroups $G=\cup G_i$ where each $G_i$ has straight finite decomposition complexity. Then, $G$ has straight finite decomposition complexity.
\end{proposition}

\begin{proof} We equip $G$ with a proper left-invariant metric. Let $R_1<R_2<\cdots$ be given. Since the metric is proper, there is some $G_i$ that contains $B_{R_1}(e)$. Then, the decomposition of $G$ into cosets of $G_i$ is $R_1$-disjoint and each coset is isometric to $G_i$, which is assumed to have sFDC.
\end{proof}

Suppose that $G$ is a group acting by isometries (on the left) on the metric space $X$. For $R> 0$, the $R$-{\em coarse stabilizer} of $x\in X$ is the set $\{g\in G\mid d(g.x,x)<R\}$, see \cite{BD1,GTY1}.  

\begin{proposition}\label{iso-sfdc} Let $X$ be a metric space with sFDC. Suppose that $\sX=\{X_\alpha\}$ is a family of metric spaces such that for each $\alpha$ there is an isometry $\varphi_\alpha:X_\alpha\to X$. Then, $\{X_\alpha\}$ has sFDC as a family.
\end{proposition}

\begin{proof} Let $R_1\le R_2\le\cdots$ be given. Since $X$ is assumed to have sFDC we can find $n$ and families $\{X\}=\sY^0,\sY^1,\sY^2,\ldots,\sY^n$ so that $\sY^{i-1}\xrightarrow{R_i} \sY^{i}$, and $\sY^n$ is bounded. Define $\widetilde{\sY^0}=\{f_\alpha^{-1}(X)\}_{\alpha}$. For $i\ge 1$ define $\widetilde{\sY^i}=\{f^{-1}_\alpha(Y)\mid Y\in \sY^i, \alpha\}$. Now, any element of $\widetilde{\sY^0}$ is equal to $X_\alpha$ for some $\alpha$. Then, since $X=Y^0\cup Y^1$ with each $Y^i$ decomposing as an $R_i$ disjoint union of sets $Y^{ij}\in\sY^i$, we see that $X_\alpha=f^{-1}_\alpha(Y^0)\cup f_\alpha^{-1}(Y^1)$. Then, $f_\alpha^{-1}(Y^i)=\sqcup f_\alpha^{-1}(Y^{ij})$, with $f^{-1}(Y^{ij})\in \widetilde{\sY_1}$. Finally if $x\in f_\alpha^{-1}(Y^{ij})$ while $x'\in f^{-1}_\alpha(Y^{ij'})$ with $j\neq j'$, then $d(x,x')=d(f_\alpha(x),f_\alpha(x'))\ge R$. Thus $\widetilde{\sY^0}\xrightarrow{R_1}\widetilde{\sY^1}$. A similar argument shows that $\widetilde{\sY^{i-1}}\xrightarrow{R_i}\widetilde{\sY^{i}}$. 

Finally, we show that $\widetilde{Y^n}$ is uniformly bounded. By assumption $\sY^n$ is uniformly bounded. Thus, there is some $D>0$ so that $\diam(Y)\le D$ for every $Y\in\sY^n$. Now, suppose that $Y'\in\widetilde{\sY^n}$. Then, $Y'=f_\alpha^{-1}(Y'')$ for some $Y''\in \sY^n$ and some $\alpha$. Since $f_\alpha$ is an isometry, $\diam(Y')\le D$.

\end{proof}

\begin{proposition}\label{WR-iso} Let $G$ be a countable group in a proper left-invariant metric as in Section 2. Suppose $G$ acts by isometries on the metric space $X$ on the left. Then, for each $\gamma\in G$, there is an isometry $W_R(x)\to W_R(\gamma x)$ given by $\g\mapsto \gamma g\gamma^{-1}$. \end{proposition}

\begin{proof}
First we observe that $g\to \gamma g\gamma^{-1}$ is a bijection from $W_R(x)$ to $W_R(\gamma x)$ for each $\gamma$:

\begin{tabular}{ccc}
$g\in W_R(x)$ & $\iff$ & $d(gx,x)< R$ \\ 
 & $\iff$ & $d(\gamma gx, \gamma x)< R$ \\ 
 & $\iff$ & $d(\gamma g\gamma^{-1}(\gamma x),\gamma x)< R$ \\ 
 & $\iff$ & $\gamma g\gamma^{-1}\in W_R(\gamma x).$
\end{tabular} 

Now, for each $\gamma$ this map is an isometry. Indeed, if $g,h\in W_R(x)$, then $d(g,h)=\|g^{-1}h\|=\|\gamma g^{-1}\gamma^{-1}\gamma h\gamma^{-1}\|=d(\gamma g\gamma^{-1},\gamma h\gamma^{-1})$.
\end{proof}

\begin{proposition}\label{coarse-stab} Let $G$ be a countable group in a left-invariant proper metric acting on a metric space $X$ by isometries. Suppose that $X$ has straight finite decomposition complexity. If there is a $x_0\in X$ so that for every $R>0$ the $R$-coarse stabilizer of $x_0$ has straight finite decomposition complexity, then $G$ has straight finite decomposition complexity. 
\end{proposition}

\begin{proof} We follow the reasoning of \cite[Proposition 3.2.3]{GTY1}. 

We may restrict our attention to the subset $G.x_0$ of $X$. We wish to apply the fibering theorem (Theorem \ref{Fibering}). To this end, we define $\pi: G\to G.x_0$ by $\pi(g)=g.x_0$. It is shown by Guentner, Tessera and Yu \cite[Lemma 3.2.2]{GTY1} that this map is uniformly expansive. Thus, it remains only to show that $\pi^{-1}(\sV)$ has sFDC for each bounded family $\sV$ in $G.x_0$.

Let $\sV$ be a bounded family in $G.x_0$. Let $R$ be so large that $\diam(V)\le R$ for all $V\in\sV$. Thus, (with $B_R(\cdot)$ denoting the open ball of radius $R$), we have that for each $V\in\sV$ there is a $\gamma_V\in G$ such that $B_R(\gamma_V.x_0)$ contains $V$. Next, we observe that $\pi^{-1}(B_R(\gamma_V.x_0))=W_R(\gamma_V.x_0)$. Thus the family $\pi^{-1}(\sV)$ consists of sets that are themselves subsets of sets in the family $\{W_R(\gamma_V.x_0)\}$.  

Using the assumption that $W_R(x_0)$ has sFDC, we apply Proposition \ref{iso-sfdc} and Proposition \ref{WR-iso} to conclude that the family $\{W_R(\gamma.x_0)\}_{\gamma\in G}$ has sFDC. Thus, by restricting the families $\sY^i$ from the realization of sFDC for the family $\{W_R(\gamma.x_0)\}_\gamma$ to subsets, we find that $\pi^{-1}(\sV)$ has sFDC whenever $\sV$ is a bounded family in $G.x_0$. 

Thus, $G$ has sFDC by Theorem \ref{Fibering}.
\end{proof}

\begin{corollary} 
The following results easily follow from this proposition.
\begin{enumerate}
	\item sFDC is closed under group extensions.
	\item sFDC is closed under free products with amalgamation and HNN extensions.
	\item sFDC is closed under finite graph products.
	\item FDC is closed under finite graph products.
\end{enumerate}
\end{corollary}

\begin{proof}
\begin{enumerate}
	\item Suppose that $1\to K\to G\xrightarrow{\phi} H\to1$ is an exact sequence of countable groups with $H$ and $K$ both having straight finite decomposition complexity. Let $G$ act on $H$ by the rule $g.h=\phi(g)h$. The $R$-coarse stabilizer is coarsely equivalent to $K$, so it has sFDC. Thus, by proposition \ref{coarse-stab}, $G$ has sFDC.
	\item This follows from the Bass-Serre theory of graphs of groups. More precisely, if $G$ is an amalgamated product (or HNN extension), then there is a tree $T$ and an action of $G$ on that $T$ by isometries with vertex stabilizers isomorphic to the factors of the amalgam. The coarse stabilizers of the action will therefore have sFDC and so $G$ itself will. 
	\item This follows from parts (1) and (2) using the technique of Proposition \ref{Cor-PropA} or \cite{An-Dre}.
	\item This is immediate from the results of \cite{GTY1} using the technique of Proposition \ref{Cor-PropA} or \cite{An-Dre}.
\end{enumerate}
\end{proof}

\bibliographystyle{alpha}

\begin{thebibliography}{DKU98}

\bibitem[AD13]{An-Dre}
Yago Antol{\'\i}n and Dennis Dressen.
\newblock The {H}aagerup property is stable under graph products.
\newblock {\em Preprint}, 05 2013.

\bibitem[BBF12]{BBF}
Mladen Bestvina, Kenneth Bromberg, and Koji Fujiwara.
\newblock Constructing group actions on quasi-trees and applications to mapping
  class groups.
\newblock {\em Preprint: arXiv:1006.1939v3}, 2012.

\bibitem[BD01]{BD1}
G.~Bell and A.~Dranishnikov.
\newblock On asymptotic dimension of groups.
\newblock {\em Algebr. Geom. Topol.}, 1:57--71 (electronic), 2001.

\bibitem[BD06]{BD-Hur}
G.~Bell and A.~Dranishnikov.
\newblock A {H}urewicz-type theorem for asymptotic dimension and applications
  to geometric group theory.
\newblock {\em Trans. Amer. Math. Soc.}, 358(11):4749--4764 (electronic), 2006.

\bibitem[BD08]{BD-AD}
G.~Bell and A.~Dranishnikov.
\newblock Asymptotic dimension.
\newblock {\em Topology Appl.}, 155(12):1265--1296, 2008.

\bibitem[Bec15]{Beck}
Susan Beckhardt.
\newblock Groups acting on metric spaces with asymptotic property {C}.
\newblock {\em Preprint}, 2015.

\bibitem[Bel03]{Be}
Gregory~C. Bell.
\newblock Property {A} for groups acting on metric spaces.
\newblock {\em Topology Appl.}, 130(3):239--251, 2003.

\bibitem[DJ99]{DJ}
A.~Dranishnikov and T.~Januszkiewicz.
\newblock Every {C}oxeter group acts amenably on a compact space.
\newblock In {\em Proceedings of the 1999 Topology and Dynamics Conference
  (Salt Lake City, UT)}, volume~24, pages 135--141, 1999.

\bibitem[DKU98]{DKU}
A.~N. Dranishnikov, J.~Keesling, and V.~V. Uspenskij.
\newblock On the {H}igson corona of uniformly contractible spaces.
\newblock {\em Topology}, 37(4):791--803, 1998.

\bibitem[Dra00]{Dr00}
A.~N. Dranishnikov.
\newblock Asymptotic topology.
\newblock {\em Uspekhi Mat. Nauk}, 55(6(336)):71--116, 2000.

\bibitem[Dra08]{Dr-amalg}
Alexander Dranishnikov.
\newblock On asymptotic dimension of amalgamated products and right-angled
  {C}oxeter groups.
\newblock {\em Algebr. Geom. Topol.}, 8(3):1281--1293, 2008.

\bibitem[DS06]{Dr-Smi}
A.~Dranishnikov and J.~Smith.
\newblock Asymptotic dimension of discrete groups.
\newblock {\em Fund. Math.}, 189(1):27--34, 2006.

\bibitem[Dyk04]{Dye}
Kenneth~J. Dykema.
\newblock Exactness of reduced amalgamated free product {$C^*$}-algebras.
\newblock {\em Forum Math.}, 16(2):161--180, 2004.

\bibitem[DZ14]{DZ}
Alexander Dranishnikov and Michael Zarichnyi.
\newblock Asymptotic dimension, decomposition complexity, and {H}aver's
  property {C}.
\newblock {\em Topology Appl.}, 169:99--107, 2014.

\bibitem[Gol13]{Gold}
Boris Goldfarb.
\newblock Weak coherence of groups and finite decomposition complexity.
\newblock 07 2013.

\bibitem[Gre90]{Green}
Elisabeth Green.
\newblock {\em Graph products of groups}.
\newblock PhD thesis, University of Leeds, 1990.

\bibitem[Gro93]{Gr93}
M.~Gromov.
\newblock Asymptotic invariants of infinite groups.
\newblock In {\em Geometric group theory, Vol.\ 2 (Sussex, 1991)}, volume 182
  of {\em London Math. Soc. Lecture Note Ser.}, pages 1--295. Cambridge Univ.
  Press, Cambridge, 1993.

\bibitem[GTY12]{GTY2}
Erik Guentner, Romain Tessera, and Guoliang Yu.
\newblock A notion of geometric complexity and its application to topological
  rigidity.
\newblock {\em Invent. Math.}, 189(2):315--357, 2012.

\bibitem[GTY13]{GTY1}
Erik Guentner, Romain Tessera, and Guoliang Yu.
\newblock Discrete groups with finite decomposition complexity.
\newblock {\em Groups Geom. Dyn.}, 7(2):377--402, 2013.

\bibitem[Gue14]{Guentner-Survey}
Erik Guentner.
\newblock Permanence in coarse geometry.
\newblock In {\em Recent progress in general topology. {III}}, pages 507--533.
  Atlantis Press, Paris, 2014.

\bibitem[KY12]{KY}
Gennadi Kasparov and Guoliang Yu.
\newblock The {N}ovikov conjecture and geometry of {B}anach spaces.
\newblock {\em Geom. Topol.}, 16(3):1859--1880, 2012.

\bibitem[PP09]{Pol-Pol}
El{\.z}bieta Pol and Roman Pol.
\newblock A metric space with the {H}aver property whose square fails this
  property.
\newblock {\em Proc. Amer. Math. Soc.}, 137(2):745--750, 2009.

\bibitem[Roe03]{Ro03}
John Roe.
\newblock {\em Lectures on coarse geometry}, volume~31 of {\em University
  Lecture Series}.
\newblock American Mathematical Society, Providence, RI, 2003.

\bibitem[Roe05]{Ro05}
John Roe.
\newblock Hyperbolic groups have finite asymptotic dimension.
\newblock {\em Proc. Amer. Math. Soc.}, 133(9):2489--2490 (electronic), 2005.

\bibitem[Smi06]{Smi}
J.~Smith.
\newblock On asymptotic dimension of countable abelian groups.
\newblock {\em Topology and its Applications}, 153(12):2047 -- 2054, 2006.

\bibitem[Tu01]{Tu}
Jean-Louis Tu.
\newblock Remarks on {Y}u's ``property {A}'' for discrete metric spaces and
  groups.
\newblock {\em Bull. Soc. Math. France}, 129(1):115--139, 2001.

\bibitem[Wri12]{Wright}
Nick Wright.
\newblock Finite asymptotic dimension for {${\rm CAT}(0)$} cube complexes.
\newblock {\em Geom. Topol.}, 16(1):527--554, 2012.

\bibitem[Yu98]{Yu98}
Guoliang Yu.
\newblock The {N}ovikov conjecture for groups with finite asymptotic dimension.
\newblock {\em Ann. of Math. (2)}, 147(2):325--355, 1998.

\bibitem[Yu00]{Yu00}
Guoliang Yu.
\newblock The coarse {B}aum-{C}onnes conjecture for spaces which admit a
  uniform embedding into {H}ilbert space.
\newblock {\em Invent. Math.}, 139(1):201--240, 2000.

\end{thebibliography}

\end{document}